\documentclass[11pt,
]{amsart}
\usepackage{
amssymb, graphicx
}
 
\newcommand{\no}[1]{#1}

\renewcommand{\no}[1]{}

\no{\usepackage{times}\usepackage[
subscriptcorrection, slantedGreek, 
nofontinfo]{mtpro}
\renewcommand{\Delta}{\upDelta}
}
\date{\today}

\numberwithin{equation}{section}%

\newtheorem{theorem}{Theorem}[section]
\newtheorem{proposition}{Proposition}[section]
\newtheorem{lemma}{Lemma}[section]
\newtheorem{definition}{Definition}[section]

\theoremstyle{definition}
\newtheorem{remark}{Remark}[section]

\DeclareMathOperator{\supp}{supp}

\DeclareMathOperator{\WF}{WF}
\DeclareMathOperator{\WFA}{WF_A}

\newcommand{\eps}{\varepsilon}

\newcommand{\R}{{\bf R}}

\renewcommand{\r}[1]{(\ref{#1})}
\newcommand{\PDO}{$\Psi$DO}
\newcommand{\be}[1]{\begin{equation}\label{#1}}
\newcommand{\ee}{\end{equation}}

\renewcommand{\d}{\mathrm{d}}

\renewcommand{\i}{\mathrm{i}}

\title[Support theorems for the Light Ray transform]{Support theorems for the Light Ray transform\\ on analytic Lorentzian manifolds}

\author[P. Stefanov]{Plamen Stefanov}
\address{Department of Mathematics, Purdue University, West Lafayette, IN 47907}
\thanks{Partly supported by a NSF  Grant DMS-1301646}
\begin{document}
\begin{abstract}
We study the weighted  ray transform  of integrating functions on  a Lorentzian manifold over lightlike geodesics. We prove support theorems  if the manifold and the weight are analytic. 
\end{abstract} 
\maketitle

\section{Introduction} 
Let $g$ be a Lorentzian metric with signature $(-,+,\dots,+)$ on the manifold $M$ of dimension $1+n$, $n\ge2$. Light-like geodesics $\gamma(s)$ (also called null geodesics) are the solutions of the geodesic equation $\nabla_s \dot\gamma=0$ for which $g(\dot\gamma,\dot\gamma)=0$. There is no canonical unit speed  parameterization as in the Riemannian case as discussed below. For some fixed choice of it, we define the weighted light ray transform $L_\kappa f$ of a function (or a distribution) $f$ on $M$  by
\be{1.1}
L_\kappa f(\gamma) = \int  \kappa(\gamma(s), \dot\gamma(s)) f(\gamma(s))\,\d s,
\ee
where $\gamma$ runs over all null geodesics. Here $\kappa$ is a weight function, positively homogeneous in its second variable of degree zero, which makes it parameterization independent.   
When $\kappa=1$, we use the notation $L$. 
Conditions for $\supp f$ and the interval of definition of the geodesics will be specified below but in all cases, the integration is carried over a compact interval. This transform appears in the study of hyperbolic equations when we want to recover a potential term from boundary or scattering information, see, e.g., \cite{MR1004174,Ramm-Sj, Ramm_Rakesh_91,waters2014stable, watersR_2013, Salazar_13, Aicha_15} for time dependent coefficients,  and also     \cite{BellassouedDSF, Carlos_12} for time-independent ones.  It belongs to the class of the restricted X-ray transforms  since the complex of curves is restricted to the lower dimension manifold $g(\dot\gamma, \dot\gamma)=0$.  

Our goal is to study the invertibility of $L_\kappa $, including its microlocal invertibility. While the methods we develop could be used to study stable recovery of the ($C^\infty$) spacelike wave front set of $f$, we concentrate our attention here on support theorems for analytic metrics and weights. 
  In \cite{MR1004174},   the author showed that if $g$ is the Minkowski metric, and if $f(t,x)$ is supported in a cylinder $\R\times B(0,R)$ and has tempered growth in the time variable, then $L f$ determines $f$ uniquely, see also \cite{Ramm-Sj}. The proof was based on the fact that $L f$ recovers the Fourier transform $\hat f$ of $f$ (w.r.t.\ all variables) in the spacelike cone $|\tau|<|\xi|$ in a direct (and stable) way and since $\hat f(\tau,\xi)$ is analytic in the $\xi$ variable (with values distributions in the $\tau$ variable), then one can fill in the missing cone by analytic continuation in the $\xi$ variable. It is easy to see that there is no stable way to recover $\hat f$ in the timelike cone $|\tau|>|\xi|  $ (true also  in the most general Lorentzian case, see next paragraph) thus $L$ has a high degree of instability, see also \cite{Begmatov01}. From a physical point of view, this could be expected: we can recover all ``signals'' moving slower than light, and we should not expect to recover those moving faster than light; and the latter should not exist anyway expect for possible group velocity faster than light. 

When the metric is not flat, it is fairly obvious that $L_\kappa f$ cannot ``see'' the wave front $\WF(f)$ in the timelike cone, this just follows from the inspection of the wave front of the Schwartz kernel of $L_\kappa $, see also  Theorem~\ref{thm_M} for the Minkowski case. Recovery of $\WF(f)$ in the spacelike cone is far less obvious and certainly requires some geometric assumptions like no conjugate points or existence of a foliation of strictly convex surfaces, as we explain below.  One possible approach 
is to analyze  the normal operator $L_\kappa 'L_\kappa $ as in   \cite{Greenleaf-Uhlmann, Greenleaf_Uhlmann90, Greenleaf_UhlmannCM}. That operator is in the $I^{p,l}$ class of  \PDO s with singular kernels, which are Fourier Integral Operators (FIOs), in fact, see \cite{Greenleaf-U_90} and the references there. The analysis of $L_\kappa 'L_\kappa $  in the Minkowski case for $n=2$ is presented in \cite{Greenleaf-Uhlmann, Greenleaf_Uhlmann90, Greenleaf_UhlmannCM} as an example illustrating a much more general theory. Applying the $I^{p,l}$ calculus to get more refined microlocal results however requires the cone condition which cannot be expected to hold on general Riemannian manifolds due to the lack of symmetry, as pointed out in \cite{Greenleaf_UhlmannCM}. 
An alternative approach to recover the $C^\infty$ spacelike singularities can be found in \cite{LOSU-strings}.

Our main result is support theorems and injectivity of $L_\kappa $ for analytic metrics and weights (on analytic manifolds $M$). It can be viewed as an extension of the classical Helgason support theorem for Radon transforms in the Euclidean space \cite{Helgason-Radon}.  
 We use analytic microlocal arguments. Such techniques go back to \cite{BQ1,BQ,Boman-Helgason}. In \cite{BQ1}, the authors prove support theorems for Radon transforms (with flat geometry) and  analytic weights. In \cite{BQ}, they study ``admissible line complexes'' in $R^{1+2}$ with analytic weights, and type III there includes a weighted version of $L_\kappa $ in the Minkowski case. 
 Their arguments however are based on the calculus of the analytic FIOs as an analytic version of the $C^\infty$ analysis in \cite{Greenleaf-Uhlmann}. Such a generalization does not exist to the best of the author's knowledge. Even the analytic \PDO\ calculus is quite delicate already, see, e.g., \cite{Treves}, and an analytic version of the FIO calculus, including the $I^{p,l}$ one, would pose even more challenges. Support theorems for the geodesic transforms on simple analytic manifolds have been proved with analytic microlocal  techniques in \cite{Venky09,SV} and related results; even for tensor fields in \cite{FSU,  SU-JAMS, SU-AJM}. A breakthrough was made by Uhlmann and Vasy in \cite{UV:local}; who proved a support theorem in the Riemannian case near a strictly convex point of a surface in dimensions $n\ge3$ without the analyticity condition. The X-ray transform is assumed to be zero on all geodesics close to tangent ones to the surface at that point, and $f$ is a priori supported on the concave side.  Their arguments are based on application of the scattering calculus \cite{Melrose94} and the $n\ge3$ assumption is needed to guarantee ellipticity in a neighborhood of the point.

The approach we propose is simpler and avoids all the difficulties related to the singularities of the symbol of $L_\kappa 'L_\kappa $: we form smooth timelike surfaces foliated by  lightlike geodesics  over which one can compute a weighted Radon transform $R$ by just applying Fubini's theorem. This reduces the problem to a microlocal inversion of that (non-restricted) Radon transform known on an open set of surfaces, which in the smooth case is doable with classical microlocal techniques going back to Guillemin \cite{Guillemin85,GuilleminS}.  Analytic singularities can be resolved by the local Radon transform, as well \cite{ZH_surfaces}.  
On the other hand, this approach does not allow us to analyze the lightlike singularities, where some form of the $I^{p,l}$ calculus would still be needed. 
In the proof of Theorem~\ref{thm_LC} those surfaces do not appear explicitly but they can be thought of as  the level surfaces of the phase function $\phi$.   
One would expect to  be able to do the analytic microlocal inversion by treating $R'R$ as an analytic \PDO\ but it is not clear how to do that to obtain purely local results due to the delicate nature of cut-offs allowed in that calculus. Instead, we use the analytic stationary phase approach by Sj\"ostrand \cite{Sj-A} already used by the author and Uhlmann in \cite{SU-AJM},  see also \cite{FSU,Venky09,SV}.

As a simple example illustrating the reduction of the restricted ray transform $L_\kappa $ to a classical Radon transform $R$, consider the Minkowski case. Light geodesics are given then by the lines parallel to $(1,\theta)$, with $|\theta|=1$.  Every timelike plane (with a normal $\nu=(\nu_t,\nu_x)$ such that $|\nu_t|<|\nu_x|$) can be represented easily as a foliation of light rays. If $L_\kappa f\in C^\infty$ (or analytic), then so is $Rf$ on the open manifold of those planes. Then we have to invert microlocally the classical Radon transform, which is well known. 
This argument still works if we introduce a weight in $L_\kappa $ and/or know $L_\kappa f$ localized to an open set of light rays only, see Theorem~\ref{thm_C}.

Finally, we notice that some global conditions on the geodesic flow are clearly needed for 
microlocal inversion, even in the spacelike cone. If $g = -\d t^2+ h_{ij}(x)\d x^i \d x^j$, where $h$ is a Riemannian metric on a bounded domain, then $L_\kappa $, restricted to $t$-independent function reduces to the geodesic X-ray transform $X$. It has been shown recently \cite{SU-caustics, MonardSU14} that when $n=2$, $Xf$ recovers $\WF(f)$ in a stable way if and only if there are no conjugate points. When $n\ge3$, the no conjugate points condition is sufficient \cite{SU-JAMS, SU-AJM} and there are examples of metrics of product type for which it is necessary, by the 2D results in \cite{SU-caustics, MonardSU14}. On the other hand, the support theorem in \cite{UV:local} provides global uniqueness and stability under another type of condition:   existence of a foliation by strictly convex surfaces (conjugate points may exist). This implies stable invertibility when $n\ge3$  without analyticity assumptions. We assume the foliation condition but in contrast to \cite{UV:local}, here $n=2$ is allowed since for our purposes, full ellipticity (in all directions at a point) is not needed; only ellipticity at directions conormal to the foliation suffices. Also, full ellipticity does not hold in the Lorentzian case since the timelike singularities are invisible. 
On the other hand, we require $g$ and $\kappa$ to be analytic. One would expect support theorems under the no-conjugate points assumption as well but that remains an open question.

\textbf{Acknowledgments.} The author thanks Manuel Guti\'errez for helpful discussions about  Lorentzian geometry and to Gunther Uhlmann for numerous discussions about  Integral Geometry and Inverse Problems.

\section{Main results} 
\subsection{Support theorems for the Minkowski spacetime} 
Let $g = -\d t^2+ (\d x^1)^2+\dots +(\d x^n)^2$ be the Minkowski metric in $\R^{1+n}$. 
Future pointing lightlike geodesics are given by $s\mapsto (t+s,x+s\theta)$ with $|\theta|=1$. They can be reparameterized by shifting and rescaling $s$.  Note that the notion of ``unit'' speed is not invariantly defined under Lorentzian transformations but in a fixed coordinate system, the scaling parameter $1$ (i.e., $\d t/d s=1$) is a convenient choice. 
Set
\be{Ldef}
Lf(x,\theta) = \int f(s, x+s\theta)\,\d s,\quad x\in \R^{n}, \,\theta\in S^{n-1}.
\ee
This definition is based on parameterization of the lightlike geodesics (lines) by their point of intersection with $t=0$ and direction $(1,\theta)$. We will use the notation
\be{g}
\ell_{x,\theta}(s) = (s,x+s\theta). 
\ee
The parameterization $(x,\theta)$ defines a natural topology and a manifold structure of the set of the future pointing lightlike geodesics. 
 
Given  a weight $\kappa\in C^\infty(\R\times\R^n\times S^{n-1})$, we can define the weighted version $L_\kappa$ of $L$ by
\[
L_\kappa f(x,\theta) = \int   \kappa(s, x+s\theta,\theta)  f(s, x+s\theta)\,\d s,\quad x\in \R^{n}, \,\theta\in S^{n-1}.
\]

In the terminology of relativity theory, vectors $v=(v^0,v')$ satisfying $|v_0|<|v'|$ (i.e., $g(v,v)>0$) are called \textit{spacelike}. The simplest example are vectors $(0,v')$, $v'\not=0$. 
Vectors with $|v_0|>|v'|$ (i.e., $g(v,v)<0$) are \textit{timelike}; an example is $(1,0)$ which points along the time axis. 
\textit{Lightlike} vectors are those for which we have equality: $g(v,v)=0$. For covectors, 
the definition is the same but we replace $g$ by $g^{-1}$, which is consistent with the operation of raising and lowering the indices.  
Surfaces with timelike normals (which are covectors) are spacelike, and vice versa.


\begin{definition}\label{def_slow}
Let $K$ be a subset of the Minkowski spacetime. We say that $K$ expands with speed less than one if
\be{supp}
K\subset \left\{ (t,x);\; |x|\le c|t|+R  \right\}\quad \text{for some $0<c<1$, $R>0$}.
\ee
\end{definition}

Condition \r{supp} is easily seen to be invariant under Lorentz transformations. Also, it does not require $\supp f$ to be compact.

The terminology we used is a bit ambiguous. What we actually mean  is that the cross-section of $K$ with any plane $t=\text{const.}$ is a bounded set contained in a ball expanding with a speed less than one. If $\partial K$ is smooth, we do not really require it to be timelike.  

In the Minkowski spacetime, we have the following support theorem.

\begin{theorem}\label{thm_linesM}
Let $f\in\mathcal{D}'(\R^{1+n})$ be so that $\supp f$ expands with a speed less than one. 
Let $\ell_{x_0,\theta_0}$ be a fixed lightlike line in the Minkowski spacetime 
and let  $U\ni(x_0,\theta_0)$ be an open and connected subset of $\R^n\times S^{n-1}$. Let $ \kappa(t,x,\theta)$ be analytic and non-vanishing for  $(t,x)$ near $\supp f$ so that   $(x-t\theta,\theta)\in U$.   

If $L_\kappa f(x,\theta)=0$   in $U$ and if $\ell_{x_0,\theta_0}$ does not intersect $\supp f$, then none of the lines $\ell_{x,\theta}$, $(x,\theta)\in U$, does. 
\end{theorem}

\subsection{Support theorems on  Lorentzian manifolds} 
Let $(M,g)$ be a Lorentzian manifold.  Light-like (null) geodesics are defined as the geodesics $\gamma(s)$ for which $g(\dot\gamma,\dot\gamma)=0$. They exist at least locally by the ODE theory.  
There is no canonical parameterization since for any linear transformation of the $s$ variable $\sigma(s) = as+b$, $a\not=0$, $\gamma\circ\sigma$ is still a null geodesic. Moreover, $a$ and $b$ may change from geodesic to geodesic. Let $S$ be a spacelike surface near a fixed lightlike geodesic  $\gamma_0(s)$, intersecting $S$ for $s=0$. Then we can parameterize the lightlike geodesics in some neighborhood of $\gamma_0(0)\in S$ close to $\gamma_0$ with directions close to $\dot\gamma(0)$ with initial points $x$ on $S$ and initial lightlike directions $v$ at $x$ pointing in the direction of $\dot\gamma_0$. A choice of the scaling of the parameter $s$ along each $\gamma(s)$ can be fixed by requiring $\gamma(0)\in S$ and requiring the normal component of $\dot\gamma$ on $S$ to be a given negative function, for example $-1$. If that function is smooth/analytic when $S$ is smooth/analytic, we call the parameterization smooth/analytic. This property does not depend on the choice of $S$ and also defines a topology and a smooth/analytic structure  of the lightlike geodesics defined on a fixed interval.   We could use a timelike surface as initial points instead. 

If $\mathcal{C}\subset M$ is closed, we call the null geodesic $\gamma(s)$ \textit{non-trapping in $\mathcal{C}$}, if $\gamma^{-1}(\mathcal{C})$ is contained in some open finite interval call it $I$ for a moment. For any local parameterization of null geodesic as above, the maximally extended null geodesic with initial points and directions close enough to $\gamma$ would leave $\mathcal{C}$ for $s$ near the ends of $I$. Some of them may return to $\mathcal{C}$ for $s\not\in I$  (even though this cannot happen to $\gamma$ but we restrict them to $I$ only. Then  we consider those geodesics a neighborhood of $\gamma$, identified with the neighborhood of the initial points and directions in that parameterization. That definition of local neighborhood is independent of the chosen parameterization and defines a topology near $\gamma$ (restricted to $I$). 
For any such choice of the parameterization, we then define $L_\kappa f$ locally by \r{1.1} for any $f\in C_0^\infty$, with $s$ restricted to $I$.
 A different analytic parameterization would change $L_\kappa f$ (in a trivial way) but it will not change its property to be smooth or analytic, or zero. 
 
  We do not want to assume that $f$ is compactly supported  but we always assume that 
 we integrate over a set of light geodesics non-trapping in  $\supp f$.  Then locally, we may cut, a smooth  $f$ in a smooth way to make it compactly supported without changing $L_\kappa f$ near that geodesic. This reduces the local analysis to compactly supported functions. An example is a function supported in a cylinder $|x|\le R$ in the Minkowski case; or more general $f$ with $\supp f$  expanding with speed less than one, see \r{supp}. 
This allows $L_\kappa f$ to be well defined for smooth $f$ over open sets of non-trapped light geodesics and then by duality for distributions $f$. Indeed, for every  distribution $f$ in $M$ we can set locally  $L_\kappa f= L_\kappa\chi f$, with a suitable $\chi\in C_0^\infty$, and the latter makes sense by duality. In other words, near a fixed null geodesic, non-trapping in $\supp f$, it is enough to study $L_\kappa$ restricted to compactly supported distributions $f$. Based on that, to simplify the formulation of the next theorem, we assume that $\supp f$ is compact.

\begin{theorem}\label{thm_L} 
Let $(M,g)$ be an analytic Lorentizan manifold and let $\kappa$ be an analytic non-vanishing weight. 
Let $F: M \to [0,1]$ be a smooth function. Assume that $f\in \mathcal{E}'(M)$ and 


 (i)  $F^{-1}(0)\cap\supp f=\emptyset$, 
 
 (ii) $\d F\not=0$ on $\supp f$,
 
 (iii) $F^{-1}(\sigma)\cap \supp f$ is strictly lightlike-convex for all $\sigma\in[0,1]$. 
 
\noindent   
Then if $L_\kappa f(\gamma)=0$ in a neighborhood of  all null-geodesics with the property that each one is tangent to some of the surfaces $F^{-1}(\sigma)$, $\sigma\in [0,1]$,  then $f=0$ on $F^{-1}[0,1)$.
\end{theorem}

We refer to Definition~\ref{def_convex} for the notion of lightlike convexity. 
Examples of strictly lightlike-convex surfaces in the Minkowski spacetime, which cannot be spacelike at any point, include the cylinder $|x|=R$, $R>0$; more generally, the smooth part of the double cone $|x|=c|t|$, with $0<c<1$ fixed; or the hyperboloid $|x|^2=c^2t^2+C$ with $C>0$ and such a $c$. They are all timelike.  We also note that we can actually require $L_\kappa f=0$ on a suitable  submanifold of lightlike geodesics of dimension $1+n$ only, as it follows form the proof. Moreover, it is only enough to assume that $L_\kappa f$ is analytic there, even microlocally so, see Remark~\ref{rem_n}. 

To demonstrate a typical application of Theorem~\ref{thm_L}, we will point out how one can show that if in the Minkowski space time $f$ satisfies \r{supp}, and $Lf=0$, then $f=0$, which, of course
follows from Theorem~\ref{thm_linesM} as well. We choose $F(\sigma) = \sigma+\tilde c^2(t-t_0)^2-|x-x_0|^2$ with $0<c<\tilde c<1$, $\sigma>0$. The constant $\sigma$ can be rescaled to be fit in $[0,1]$. 
Choosing various $x_0$ and $t_0$ we can prove $f=0$. One can perturb the metric a little bit assuming that $\supp f$ is supported in a fixed compact set, to get examples for non-flat metrics. 

\section{Analysis in the Minkowski case} 
\subsection{Fourier Transform analysis} Let $M=\R^n$ and $g$ be Minkowski.  By the Fourier Slice Theorem, knowing the X-ray transform for some direction $\omega$  recovers uniquely $\hat f$ on $\omega^\perp$.   
More precisely, the Fourier Slice Theorem in our case can be written as
\[
\hat f|_{ \tau+\xi\cdot\theta=0 } = 
\hat f(-\theta\cdot\xi,\xi) = \int_{\R^n} e^{-\i x\cdot\xi}Lf(x,\theta) \, \d x, \quad\forall\theta\in S^{n-1}. 
\]
Here and below, we denote by $\zeta = (\tau,\xi)$ the dual variables to $z=(t,x)$. The proof is easy, see \r{5}. The union of all $( 1,\theta)^\perp$ for all unit $\theta$ is $\{|\tau|\le |\xi|\} = \Sigma_s \cup\Sigma_t $, that is easy to see. This correlates well with the theorems below. In particular, we see that knowing $\hat f(\zeta)$ for a distribution $f$ with a well defined Fourier transform, recovers $\hat f$ in  the spacelike cone uniquely and in a stable way. Under the assumption that $\supp f$ is contained in the cylinder $|x|\le R$ for some $R$ (and temperate w.r.t. $t$), one can use the analyticity of the partial Fourier transform of $f$ w.r.t.\ $x$ to extend $\hat f$ analytically to the timelike cone, as well. 
This is how it has been shown in  \cite{MR1004174} that $L$ is injective on such $f$. 

\subsection{The normal operator $X'X$}
We formulate here a theorem about the Schwartz kernel of the normal operator $N=L'L$. We will skip the proof because we will not use the theorem for our main results. One way  to obtain it is to think of $L$ as a weighted version of the X-ray transform $X$ with a distributional weight $\delta(\tau^2-|\xi|^2)$ and use the results about the weighted X-ray transform, see, e.g., \cite{SU-Duke} and allow a singular weight there. Details will appear in \cite{SU-book}, see also \cite{LOSU-strings}.

\begin{theorem}\label{thm_M}\ 

(a) 
\[
L'Lf = \mathcal{N}*f, \quad \mathcal{N}(t,x) = \frac{\delta(t-|x|) +\delta(t+|x|) }{|x|^{n-1}}.
\]

(b) 
\[
L'Lf =  C_n\mathcal{F}^{-1}\frac{(|\xi|^2-\tau^2)_+^\frac{n-3}2} {|\xi|^{n-2}} \mathcal{F}f, \quad \forall f\in \mathcal{S}(\R^{1+n}), \quad C_n:= 2\pi|S^{n-2}| . 
\]

(c) 
\[
h(\Box_+) f = C_n^{-1} |D_x|^{n-2}\Box_+^{\frac{3-n}{2}}X'Xf,
\]
where $h$ is the Heaviside function, and $\Box = \partial_t^2-\Delta_z$ and $\mathcal{F}$ is the Fourier transform. 
\end{theorem}

Above, we used the notation $s_+=\max(s,0)$ with the convention  that $s_+^0$ is the Heaviside function. 

In particular, when $n=3$, we get $\sigma(X'X) =   C_3|\xi|^{-1}h\left( |\xi|^2- |\xi_0|^2\right)$. Then
\[
h(\Box_+)f = C_3^{-1} |D_z|X'Xf.
\]
As we can expect, there is a conormal singularity of the symbol even away from $\xi=0$ living on the characteristic cone, and $X'X$ is elliptic outside it, and only there. The theorem shows that ``singularities traveling slower than light'' can be recovered.  The ones traveling faster cannot.

 \subsection{Recovery of spacelike  $C^\infty$ singularities in the Minkowski case} \label{sec_2.3}

Our first theorem says that knowing $Lf$ near a lightlike geodesic $\ell_0$ allows us to recover all spacelike singularities conormal to $\ell_0$.  We denote the conormal bundle of $\ell_0$ by $N^*\ell_0$.  Recall that the conormal singularities to $\ell_0$ contain lightlike ones, as well. This result follows from the analysis in  \cite{Greenleaf-Uhlmann, Greenleaf_Uhlmann90, Greenleaf_UhlmannCM} and the reason we present it here is to illustrate the main idea on a simpler problem where we can do explicit computations.  
  
\begin{theorem}\label{thm_C}
Let $f$ be a distribution so that $\ell(s)\not\in\supp f$  for $|s|>1/C$  with some $C$ for all  lightlike lines  $\ell$ near $\ell_0$. Let $L_\kappa f(\ell)\in C^\infty$ for $\ell$ in some neighborhood $\Gamma$ of $\ell_0$. Then $\WF(f)\cap N^*\ell_0$ contains no spacelike covectors.  
\end{theorem}

\begin{proof}  
We construct planes foliated by lightlike geodesics with a fixed direction $(1,\theta)\in \R \times S^{n-1}$. Any such plane intersects the $t=0$ plane in a $(n-1)$-dimensional plane in the $x$ space. Let the latter be $\pi_{p,\omega}= \{x\cdot\omega=p\}$, $\omega\in S^{n-1}$, $p\in \R$.  Then the plane that we denote by  $ \pi_{p,\omega,\theta}$ is the flow out of the null geodesics with  celestial direction $\theta$ originating from $\pi_{p,\omega}$, i.e.,
\be{pi1}
\pi_{p,\omega,\theta} = \{(t,x+t\theta);\; x\in \pi_{p,\omega}\}.
\ee
The same plane can be also described by the equation $(x-t\theta)\cdot\omega=p$, therefore,
\be{pi2}
\pi_{p,\omega,\theta} = \{(t,x);\; (t,x)\cdot (-\theta\cdot\omega,\omega)=p\}.
\ee
The dot product here is in the Euclidean sense, and can be also thought of as a pairing of a vector and a covector, in invariant terms. In particular, we see that the set of such planes coincides with the timelike ones and the lightlike ones as a borderline case. 

Let $\zeta^0\not=0$ be spacelike and conormal to $\ell_0$ at a point that we can always assume that to be the origin. Applying a Lorentz transformation, we can always assume that $\zeta^0 =  e^{n-1}: = (0,\dots,0,1,0)\in \R^{1+n}$ and $\ell_0= \ell_{0,e_{n}} = (s,0,\dots0,s)$.  Here and below, we use the notations $e_k$ and $e^k$ to denote vectors/covectors with all entries zero instead of the $k$-th one. 
Take the plane $\pi_0=\{(t,x);\; x^{n-1}=0\}$, conormal to $\zeta^0$. This is the plane constructed above with $\omega=\zeta^0= e^{n-1}$ and $\theta = e_{n}$ (we could have chosen any other $\theta\perp\omega$ but we chose this one because it is related to $\ell_0)$.

It is more convenient to extend the parameters $(p,\omega)$ by homogeneity. We allow $\omega$ to be non-unit and denote it by $\xi$. Then the planes $\pi_{p,\omega,\theta}$ are given by 
\be{zeta()}
z\cdot\zeta=p, \quad \zeta= (-\theta\cdot\xi,\xi). 
\ee
We will  choose a suitable analytic  family of  $(\xi,\theta)$ near $\xi=e^{n-1}$, $\theta=e_{n}$ 
  parameterized by an $n+1$-dimensional parameter   so that the map from that parameter to the      (co-)normal $\zeta$  is a local diffeomorphism. 
We  keep $\xi$ unrestrained  and let $\theta$ depend on an 1D  parameter, $q$. Then 
\be{3''}
\partial \zeta/\partial \xi_k   = (- \theta^k ,e^k ),\quad  k=1,\dots,n, \quad \text{and}\quad \partial \zeta/\partial q|_{q=0} = (- \partial \theta/\partial q\cdot\omega|_{q=0},0). 
\ee
This system of vectors is linearly independent, if and only if $\partial \theta/\partial q\cdot\omega|_{q=0} \not=0$.  
Therefore, the variation $\partial \theta/\partial q$ should be chosen not parallel to $\pi_{p,\xi}$. This leaves essentially a variation in the direction of $\xi$. 
 Based on that, we set 
 \be{theta}
 \theta(q) = (\cos q) e_{n}+(\sin q) e_{n-1}. 
 \ee
 Then  
 \be{zeta}
 \zeta(q,\xi) = (-\theta(q)\cdot\xi,\xi)
 \ee
and
 \be{4}
 \pi_{p,\xi,\theta(q)} = \{(t,x);\; (t,x)\cdot \zeta(q,\zeta)=p\}, \quad p\in\R;\; \xi\in\R^{n}\setminus 0. 
 \ee
The fact that $(q,\xi)\to\zeta$ is a local diffeomorphism is also easy to verify directly. Solving \r{zeta} for $(q,\xi)$ yields $\xi_i=\zeta_i$, $i-1,\dots,n$, and 
\be{4e}
\zeta_n\cos q+ \zeta_{n-1}  \sin q= -\zeta_0
\ee
and the latter is uniquely solvable for $q$ near $q=0$ for $\zeta$ near $e^{n-1}$; let $q=q(\zeta)$ be the solution. 
 
 We can write the defining equation also as $(t,x)\cdot\nu=\tilde p :=p/|\zeta|$ with   $\nu=\zeta(q,\xi)/|\zeta(q,\xi)|$. Then it is easy to show that with $\xi$ restricted back to unit sphere, the map $\R\times S^{n-1} \ni (q,\xi)\to  \nu\in   S^{n}$ is a local analytic  diffeomorphism near $p=0$. In other words, $(q,\xi)$ with $\xi$ unit, parameterizes the normal $\nu$ to \r{4} in a locally diffeomorphic way.

By the support assumption of the theorem, there exists a neighborhood $U$ of $(0,e_{n})$ and $A>0$, so that all lightlike geodesic issued from $U$ leave $\supp f$ for $|s|\ge A$. Take a smooth function $\chi(x,\theta)$ supported in $U$ equal to $1$ near $(0,e_{n})$.  
	Since $L_\kappa f\in C^\infty$, we have $\chi L_\kappa f|_{x\in  \pi_{p,\xi,\theta(q)}, \theta=\theta(q)}\in C^\infty$, as well. Integrate $[\chi L_\kappa f](x,\theta(q))$ with respect to $x$ on the plane  $ \pi_{p,\xi}$  to get by Fubini's theorem:
\be{4r}
Rf(\pi_{p,\xi,\theta(q)} ) := \int_{ \pi_{p,\xi,\theta(q)}} \chi   \kappa f  \,\d\mu_{ {p,\xi,\theta(q)}}\in C^\infty
\ee
for some measure   analytically depending on $(p,q,\xi)$, i.e. an analytic  and positive multiple of the Euclidean measure on each plane. Above,   the integral is taken over a compact set; moreover, we can cut $f$ to a compactly supported distribution away from where we integrate without affecting the integral. Therefore, $f$ is in the microlocal kernel of the weighted Radon transform $R$ with a weight not vanishing at $(t,x)=0$ on the plane $\pi_0$. This allows us to apply $R'$ to get an elliptic \PDO\ of order $-2$, see, e.g., \cite{BQ}. Therefore, $f$ is microlocally smooth at $(0,\zeta^0)$ as claimed.  
\end{proof} 

\begin{remark}\label{rem_n}
Note that we only needed to know that $L_\kappa f(x,\theta)$ vanishes (being microlocally smooth in a certain cone would suffice)  for $\theta=\theta(q)$ only with $|q|\ll1$; i.e., we require knowledge of a restricted version of the already restricted $L$. This is similar to the known fact that in the Euclidean space we can invert the X-ray transform by ``slicing'' $\R^n$ into 2D planes. We could have proven our results in the Minkowski spacetime in $1+2$ dimensions only and then extended it to any dimension $1+n\ge 1+2$. The same remarks applies to the analytic case below but we need to know that $L_\kappa f$ is microlocally analytic (instead of just smooth) in some conic set. Even in the Lorentzian case, we still need to know $L_\kappa$ restricted to a certain an $(1+n)$-dimensional submanifold of geodesics. 
\end{remark}

\subsection{Recovery of analytic spacelike  singularities in  the Minkowski case}

We show first that we can recover all spacelike analytic singularities of $f$ conormal to  the lightlike lines along we integrate. For a definition of the analytic wave front set $\WFA(f)$, we refer to \cite{Sj-A} and \cite{Treves}. 

\begin{lemma}\label{lemma3.1}
Let $f\in\mathcal{D}'(\R^{1+n})$ and let $\ell_{x_0,\theta_0}$ be a fixed lightlike line so that  $\ell_{x,\theta}(s)\not\in\supp f$ for $|s|\ge 1/C$  with some $C$ for all  $(x,\theta)$ near $(x_0,\theta_0)$. Let $ \kappa(s,x+s\theta)$ be analytic and non-vanishing for  those   $s,x,\theta$. If $L_\kappa f(x,\theta)=0$     near $(x_0,\theta_0)$, then $N^*\ell_{x_0,\theta_0}\cap \WFA(f)$ contains no spacelike covectors. 
\end{lemma} 

\begin{proof}
One would expect the proof to be a complete analog of that of Theorem~\ref{thm_C} but that proof involves smooth cutoffs along the planes we integrate over. We cannot do this in our case because that would destroy the analyticity of the weight. On the other hand, we need the localization because we  know that $L_\kappa f=0$ near $\ell_{x_0,\theta_0}$ only. 

We use the local coordinates  in the proof of Theorem~\ref{thm_C}, where $\zeta^0=e^{n-1}$, $\theta_0 = e_n$, $z_0=0$. 

Let $\chi_N\in C_0^\infty(\R^n)$ be with  support in $B(0,\eps)$, $\eps>0$,  with  $\chi_N=1$ near $x_0=0$ so that 
\be{A1}
|\partial_{x}^\alpha\chi_N|\le (CN)^{|\alpha|}, \quad \text{for $|\alpha|\le N$}, 
\ee
see \cite{Treves}. Then  for $0<\eps\ll1$, $\lambda>0$, and $\theta$ close to $\theta_0=e_n$,
\[
0= \int e^{\i \lambda x\cdot\xi} (\chi_N L_\kappa f)(x,\theta)\,\d x= 
\iint e^{\i\lambda x\cdot\xi} \chi_N (x)
  \kappa(s,x+s\theta,\theta) f(s,x+s\theta)\,\d s\,    \,\d x. 
\]
If $(1,\theta)\cdot\zeta=0$ with $\zeta=(\tau,\xi)$, then $x\cdot\xi= (t,x+t\theta)\cdot\zeta$. Make the change of variables $x+s\theta\mapsto x$ in the integral above to get
\be{5}
\begin{split}
0 &= \int e^{\i \lambda x\cdot\xi} (\chi_N L_\kappa f)(x,\theta)\,\d x\\
& = 
\iint e^{\i\lambda (t,x)\cdot\zeta} \chi_N (x-t\theta)   
\kappa(t,x,\theta) f(t,x)\,\d t\,    \,\d x, \quad \text{if $(1,\theta)\cdot\zeta=0$}. 
\end{split}
\ee
For $\xi\in \R^n$, choose $\theta=\theta(q)$ as in \r{theta}, and set   
$\zeta=  (-\theta(q) \cdot\xi,\xi)$ as in   \r{zeta()}.  Then the orthogonality condition in \r{5} is satisfied. 
To connect this with the analysis in section~\ref{sec_2.3}, notice that we can get the same result by taking the Fourier transform  $\mathcal{F}_{p\to\lambda }$ in \r{4r}, with $\chi=\chi_N$. Choose now $q=q(\zeta)$ as in \r{4e}. The orthogonality condition still holds and we have 
\be{5.2}
\int e^{\i\lambda z\cdot\zeta }  a_N(z,\zeta) f(z)\,\d z =0 \quad \text{near $\zeta=e^{n-1}$},
\ee
where $ a_N= \chi_N(x-t\theta(q)
)\kappa(t,x,\theta(q))  $, with $q=q(\zeta)$, is analytic and elliptic near $(z,\zeta) = (0,e^{n-1})$ (but not analytic away from some neighborhood of it) and satisfies pseudo-analytic estimates of the type \r{A1}. 

We will apply the complex stationary phase method of Sj\"ostrand \cite{Sj-A} similarly to the way it was applied in \cite{KenigSU} to the partial data Calder\'on problem and in \cite{FSU, SU-AJM} to integral geometry ones. 

Fix $0<\delta\ll\eps$, see \r{A1}.  With some $w\in\R^{1+n} $, $\eta\in \mathbf{R}^{1+n}$ close to $w=0$, $\eta=e^{n-1}$, multiply the l.h.s.\ of \r{5} by
\be{chi_delta}
e^{ \i \lambda(\i ( \zeta-\eta)^2/2 -  w\cdot\zeta)    } 
\ee
and integrate w.r.t.\ $\zeta$ in the ball $|\zeta-\eta|< \delta$ to get 
\be{6}
\int_{ |\zeta-\eta|< \delta}\int e^{\i\lambda\Phi(w,z,\zeta,\eta)}  a_N(z,\zeta)f(z)\,\d z\, \d \zeta=0,
\ee
where 
\[
\Phi = (z-w)\cdot\zeta +  \i (\zeta-\eta)^2/2.
\]
We split the $z$ integral \r{6} into two parts: over $\{z;\; |z-w|<\delta/2\}$ and then over the complement of that set. Since $|\Phi_\zeta|$ has a ($\delta$-dependent) positive lower bound for $|z-w|\ge\delta/2$, we can integrate in the outer  integral in \r{6}  by parts w.r.t.\ $\zeta$, see, e.g., \cite{FSU, SU-AJM} using \r{A1} and the fact that on the boundary $|\zeta-\eta|=\delta$, the factor $e^{\i\lambda \Phi}$ is exponentially small with $\lambda$. We then get
\be{7}
\Big|   \iint_{|z-w|<\delta/2, \,|\zeta-\eta|\le \delta }  e^{\i\lambda\Phi(w,z,\zeta,\eta)}  a(z,\zeta )f(z)\,\d z\, 
\d \zeta     \Big| \le C(CN/\lambda)^N + CNe^{-\lambda/C}
\ee
where $ a$ equals $a$ with the $\chi_N$ factor missing, i.e., $a= \kappa(t,x,\theta(q(\zeta)) $, which is  independent of $N$ because on the support of the integrand, that factor is equal to $1$ for $\delta\ll\eps$, see \r{5.2}. Choose now $N$ so that  $N\le \lambda/(Ce)\le N+1$ to get an exponential error on the right.

The phase $\Phi$, as a function of $\zeta$, has a unique  critical point   
$
\zeta_c =\eta +\i (z-w)
$
and $|\zeta_c-\eta| \le \delta/2$ on the support of the integrand in \r{7}. 
Set 
\be{psi}
\psi(w,z,\eta) = \Phi|_{\zeta=\zeta_c}.
\ee
Therefore, 
\[
\psi = \eta\cdot (z-w) + \i |z-w|^2 -\frac{\i}2|z-w|^2 = 
\eta\cdot (z-w)+\frac{\i}2 |z-w|^2 .
\]
This is the type of phase functions that are used to test for analytic microlocal regularity.  
We apply now \cite[Theorem~2.8]{Sj-A} and the remark after it to the $\zeta$-integral in \r{7} 
     to get
\be{8}
\Big|   \int_{|z-w|<\delta/2}  e^{\i\lambda\psi(w,z,\eta)} b(w,z,\eta ,\lambda)f(z)\,\d z\, 
    \Big| \le  Ce^{-\lambda/C}.
\ee
for $(z,\eta)$ close to $(0,e^{n-1})$, 
with some classical elliptic analytic  symbol $b$ of order $0$ in the sense of \cite{Sj-A} near $(w,z,\eta)= (0,0,e^{n-1})$. In particular, the principal part of $b(0,0,e^n,\lambda)$ is $\beta  \kappa(0,0,e_{n})$ with $\beta$ an elliptic factor depending on the phase, see \cite[Theorem~2.8]{Sj-A}. 
This implies  $(0,e^{n-1})\not\in \WFA(f)$, see \cite{Sj-A}. 
\end{proof} 

\subsection{Proofs of the support theorems in the Minkowski spacetime} 
The next proposition is a unique continuation result across a timelike surface in the Minkowski case which implies Theorem~\ref{thm_linesM}. 

\begin{proposition}\label{pr_M_supp}
Let $f\in\mathcal{D}'(\R^{1+n})$ and let $\ell_{x_0,\theta_0}$ be a fixed lightlike line in the Minkowski spacetime so that  $\ell_{x,\theta}(s)\not\in\supp f$ for $|s|\ge 1/C$  with some $C$ for all  $(x,\theta)$ near $(x_0,\theta_0)$. Let $ \kappa(s,x+s\theta)$ be analytic for  those   $s,x,\theta$. 
  
If $L_\kappa f(x,\theta)=0$ near $(x_0,\theta_0)$   and if $f=0$ on one side of $S$ near $z_0$,   then $f=0$ near $z_0$. 
\end{proposition} 
\begin{proof}
Assume that $z_0\in\supp f$. Then $(z_0,\mp \nu(z_0))\in \WFA(f)$ by the  Sato-Kawai-Kashiwara Theorem, see \cite{SKK} and \cite{Sj-A}, where $\nu(z_0)$ is one of the two unit co-normals to $S$ at $z_0$. That covector is spacelike by the assumption about $S$, and is conormal to $\dot\ell_{z_0,\theta_0}(0)$. This contradicts Lemma~\ref{lemma3.1}, which completes the proof of the proposition. 
\end{proof}

\begin{proof}[Proof of Theorem~\ref{thm_linesM}] 
Fix $(x_1,\theta_1)\in U$. Let $[0,1]\ni p\to (x_p,\theta_p)$ be a continuous family in $U$  connecting $(x_0,\theta_0)$ with $(x_1,\theta_1)$.  
We can always assume that $U$ is bounded, hence $\bar U$ is compact. 
Let $\mathcal{U}$ be the set of points lying on $\ell_{x,\theta}$, $(x,\theta)\in U$.  

Choose $\tilde c\in(c,1)$, where $c$ is the constant in \r{supp}. Denote by $\tilde \ell_{x,\theta}(s) =(s,x+s\tilde c\theta) $, $(x,\theta)\in\R^n\times S^{n-1}$, the timelike geodesics with speed $\tilde c$. By \r{supp}, there exists $A>0$ so that $ \ell_{x,\theta}(s)\not\in\supp f$ for all $(x,\theta)\in \bar U$ and $|s|\ge A$, 
   and so that the same holds for $\tilde \ell_{x,\theta}(s)$ uniformly w.r.t.\ $\tilde c$ as long as $c+\mu\le \tilde c\le1$ with some fixed $\mu\in(0,1-c)$.

Let $\eps>0$ be such that   the cylinder $C_{x_0,\theta_0}: = \cup_{|y-x_0|\le\eps} \ell_{y,\theta}$  is disjoint from $\supp f$ but sill lies in $\mathcal{U}$. By the arguments above, we can assume that $\supp f$ is compact since $t$ is bounded on it along the light  lines under consideration. 
Therefore, for the cylinder $\tilde C_{x,\theta}: = \cup_{|y-x|\le\eps/2} \tilde\ell_{y,\theta}$ we have 
\be{9}
\tilde C_{x,\theta}\cap \{|t|\le A\}\subset   C_{x,\theta}\cap \{|t|\le A\}
\ee
 for every $(x,\theta)$ as long as $\tilde c$ is close enough to $c$ but still smaller than it. We require one more property for $\tilde c$ which actually refines \r{9}: 
\be{10}
 \begin{split}
 &\text{$ \forall (x,\theta)\in\bar U  $, $\forall z\in \tilde C_{x,\theta}\cap \{|t|\le A\}$, and every unit $\theta_1$ with $|\theta_1-\theta|\le \sqrt{1-\tilde c^2}$,}\\
 &\text{the lightlike line through $z$ in the direction of $(1,\theta_1)$ stays in  $C_{x,\theta}\cap \{|t|\le A\}$.}
 \end{split}
\ee
This property can be guaranteed for $1-\tilde c\ll1$ by continuity and compactness. 
 We fix such a $\tilde c$. 

Assume that the family $\{\tilde C_{x_p,\theta_p};\; p\in[0,1]\}$ has a common point with  $\supp f$. Let $p_0$ be the least $p\in[0,1]$ (which exists by compactness and continuity arguments) for which $\tilde C_{x_p,\theta_p}\cap \supp f \not=\emptyset$. Then $f=0$ in the interior of $\tilde C_{x_p,\theta_p}$ and there is a point $z^\sharp$ on its boundary which is also in $\supp f$. Let $\zeta^\sharp$ be a non-vanishing conormal to that cylinder at $z^\sharp$. After normalization, we get $\zeta^\sharp = (-c\theta_p\cdot\omega,\omega)$ for some  $\omega\in S^{n-1}$. Clearly, $\zeta^\sharp$ is spacelike. Let $\tilde \ell_{x^\sharp,\theta_p}$ be the line on the cylinder $\tilde C_{x_p,\theta_p}$   through $z_0$; then $\zeta^\sharp$ is conormal to it at $z^\sharp$. 

To apply Proposition~\ref{pr_M_supp}, we claim that there is a lightlike line at $z^\sharp$ 
normal to $\zeta^\sharp$  so that this line is still in $\mathcal{U}$. Suppose for a moment that this done. Then by 
by Proposition~\ref{pr_M_supp}, we would get that $f$ vanishes near that point, which would be a contradiction. Therefore, such a $p_0$ would not exist, and in particular, $f=0$ near $\ell_{x_1,\theta_1}$. 

To prove the claim, we are looking for a unit $\theta^\sharp$ so that $\zeta^\sharp\cdot (1,\theta^\sharp)=0$. This is equivalent to solving $(\theta^\sharp -c\theta)\cdot\omega = 0$ for $\theta^\sharp$. It is easy to see that this is always possible to do and the solution closest to $\theta$ is at its farthest distance from $\theta$ when $\omega=\pm \theta$; then $|\theta^\sharp-c\theta|=\sqrt{1-c^2}$. This shows that $|\theta^\sharp-\theta|\le \sqrt{1-c^2}$ uniformly in $\omega$. Property \r{10} then proves the claim.
\end{proof}

\section{The Lorentzian case}

\subsection{Support theorems for analytic Lorentzian manifolds} Let $(M,g)$ be a Lorentzian manifold now.  
Next theorem is an analog of Theorem~\ref{thm_C}. Since the global geometry of the null-geodesics in the Lorentzian case is non-trivial and in particular, one can have conjugate points, the assumptions are stronger.

\begin{definition}\label{def_convex}
Let $S$ be a  smooth surface near a point $z\in S$ and let $F$ be a defining function so that $S=F^{-1}(0)$ near $z$, $\d F(z)\not=0$, and declare $\{F<0\}$,  to be the ``interior'' of $M$ near $z$. Similarly, $\{F>0\}$ is the ``exterior'' of $M$ near $z$.  We say that $S$ is \textit{strictly convex} at $z$ in the direction $v\in T_zS$, if  $\nabla^2F(z)(v,v)>0$. 

We call $S$  strictly lightlike-convex if it timelike,  it is strictly convex at all lightlike $(z,v)\in TS$, and  every maximal lightlike geodesic tangent to $S$ at some point has no other common points with $S$.
\end{definition}

   Here $\nabla^2F$ is the Hessian of $F$, with $\nabla$ being the covariant derivative. This notion of convexity is equivalent to $\frac{d^2} {\d s^2}F\circ\gamma(s)<0$ for the geodesic $\gamma$ through $x$ in the direction $v$; and it is independent of the choice of $F$.

\begin{theorem}\label{thm_LC} 
Let $(M,g)$ be an analytic Lorentzian manifold. 
Let $S$ be a timelike surface near a fixed point $z_0\in S$. Let $\gamma_0$ be a lightlike  geodesic through $z_0$  tangent to $S$ at $z_0$. Assume that $S$ is strictly convex at $z_0$ in the direction of $\dot\gamma_0$, and that $\kappa$ is analytic and non-vanishing near $(z_0,\dot\gamma_0|_{z_0})$. Let $f$ be a distribution, and let $\gamma_0$ be non-trapping in $\supp f$. 
 Let $L_\kappa f(\gamma)=0$ for all lightlike geodesics  $\gamma$ near $\gamma_0$. 
If $f=0$ in the exterior of $S$ near $z_0$, then $f=0$ near $z_0$. 
\end{theorem}

\begin{proof}
By the  Sato-Kawai-Kashiwara Theorem, see \cite{SKK} and \cite{Sj-A}, that we already used in the proof of Proposition~\ref{pr_M_supp}, it is enough to prove that $f$ is microlocally analytic in the direction of the conormal $\zeta_0$ to $S$ at $z_0$. 

We follow the construction in  Proposition~\ref{pr_M_supp}.  We can consider the former proof as a linearized version of the present one, when we replace $S$ with its tangent plane at $z_0$ normal to $\zeta_0$, and the geodesics through $z_0$ by tangent lines. 

By the strict convexity assumption, 
for any lightlike geodesic $\gamma$ which is an $O(\eps)$, $\eps\ll1$, perturbation of $\gamma_0$ (in a fixed parameterization), the intersection of $\gamma$ with the interior of $S$, in any local chart has Euclidean length  $O(\sqrt\eps)$. This allows us to work in a fixed coordinate system $(t,x)$ near $z_0$. We choose a spacelike surface $S_0$ through $z_0$.  We then choose 
  semigeodesic coordinates  $(t,x)$ near $z_0=0$ normal to $S_0$, i.e., the lines $(t,x) =(t,\text{const.})$, are future pointing (the future direction being determined by $\gamma_0$) timelike geodesics normal to $S_0$ and $g$ is given locally by $-\d t^2+ g_{\alpha\beta}\d x^\alpha \d x^\beta$, see \cite{Petrov_book}. Such coordinates are constructed by taking a normal field $v$ to $S$ normalized so that $g(v,v)=-1$ and using it as initial directions of the geodesics $x=\text{const}$. We can arrange that $z_0=0$ and $\dot\gamma_0(0)=(0,e_{n})$. 

In those coordinates, future pointing geodesics near $\gamma_0$, close to $S_0$, are parameterized by their initial points $x\in S$ and the projection $\theta$ of their tangents to $TS_0$, i.e., $\gamma_{x,\theta}(s)$ is defined as the geodesic issued from $(0,x)$ with $\dot\gamma_{x,\theta}(0)=(1,\theta)$.

For  $(x,\theta)$ close to $(0,e_{n})$, 
  we then write
\[
L_\kappa f(x,\theta) = \int \kappa(\gamma_{x,\theta}(s),\dot \gamma_{x,\theta}(s))  f(\gamma_{x,\theta}(s))\,\d s.
\]

We chose $\theta=\theta(q)$ as in \r{theta} with $|q|<\eps$, where $\eps$ is the number controlling the size of $\supp\chi_N$, see \r{A1}. 
Then  for $0<\eps\ll1$,  $\lambda>0$, 
\[
\begin{split}
0 &= \int e^{\i \lambda x\cdot\xi} (\chi_N L_\kappa f)(x,\theta(q))\,\d x\\
&= 
\iint e^{\i\lambda x\cdot\xi} \chi_N (x)  \kappa(\gamma_{x,\theta(q)}(s), \dot \gamma_{x,\theta(q)}(s)) f(\gamma_{x,\theta(q)}(s) )\,\d s\,    \,\d x. 
\end{split}
\]
For every fixed $q$ near $q=0$, which fixes $\theta=\theta(q)$,  the map $(s,x)\to z =\gamma_{x,\theta}(s) $ is a local diffeomorphism near $z_0=0$ by the Implicit Function Theorem.   
Let $s^\sharp(z,\theta)$, $x^\sharp(z,\theta)$ be the inverse map. Since   $\gamma_{x,\theta}(s) = (s,x+s\theta)+O(s^2)$, we get the Taylor expansion 
\be{10q}
s^\sharp(z,\theta(q)) = t+ O(t^2), \quad   x^\sharp (z,\theta(q))= x-t \theta(q) +O(t^2),
\ee
where $z:=(t,x)$. 
Those expansions can be justified by the Implicit Function Theorem.  
Make the change of variables $(s,x)\to z$ above to get
\be{11}
0= 
\int e^{\i\lambda\phi} \chi_N (x^\sharp (z,\theta(q))) \kappa J (q,z)f(z)\,\d z
\ee
with $\phi(z,\xi,q  ) = x^\sharp (z,\theta(q))\cdot\xi$. Here, $\kappa$ is the weight in the new variables, and $J$ is the related Jacobian. 
If $g$ is Minkowski, we get $\phi = (z-t\theta(q))\cdot\xi$, which is the same function as in \r{5.2}. 

Set $\zeta=(q,\xi)$. Then $q=\zeta_0$, $\xi=\zeta' = (\zeta_1,\dots,\zeta_n)$ and 
\be{10phi}
\phi(z,\zeta  ) =   x^\sharp (z,\theta(\zeta_0))\cdot\zeta'. 
\ee

\begin{lemma}\label{lemma_phase} 
 $\det\phi_{z\zeta}(0,e^{n-1})= -1$. 
 
\end{lemma}

\begin{proof}
To compute $\phi_{x\zeta}(0,e^{n-1})$, write first (recall that $z_0=t$)
\be{15}
\phi_{\zeta_k}|_{\zeta= e^{n-1}, z_0=0} = x^\sharp(z,e_n)|_{z_0=0} = z^k, \quad k=1,\dots,n. 
\ee
Therefore, 
\[
\phi_{z^i\zeta_k}(0,e^{n-1})=\delta_i^k, \quad k=1,\dots,n, \; i=0,1,\dots,n. 
\]
Therefore, $\det\phi_{z\zeta}(0,e^{n-1})= \phi_{\zeta_0 z^0}(0,e^{n-1})$. One the other hand, the latter equals $-1$ as follows from \r{10q} and \r{theta}. 
\end{proof}

We now get from \r{11}:
\be{12}
0=\int e^{\i\lambda \phi(z,\zeta)}  a_N(z,\zeta) f(z)\, \d z=0\quad \text{near $\zeta=e^{n-1}$},
\ee
compare with \r{5.2}, with $ a_N$ elliptic and analytic near $(0,e^{n-1})$ but not for all $(z,\zeta)$. On the other hand, it satisfies a pseudo-analytic estimate of the type \r{A1}. 

Similarly to the proof of Lemma~\ref{lemma3.1}, for $w$ and $\eta$ as \r{chi_delta},  in multiply \r{12} by the factor
\[
e^{\i \lambda(\i (\zeta-\eta)^2/2 - \phi(w,\zeta) ) }
\]
and integrate w.r.t.\ $\zeta$ over the ball $|\zeta-\eta|<\delta  $ with $0<\delta\ll\eps$ to get \r{6} with 
\[
\Phi = \phi(z,\zeta)- \phi(w,\zeta)+   \i (\zeta-\eta)^2/2 .
\]
The rest of the proof follows closely those in \cite{FSU, SU-AJM}. By Lemma~\ref{lemma_phase},  $|\Phi_\xi|$ has a lower bound outside  any neighborhood of $z=w$ for $w$ localized as above and $z$ in the support of the integrand. This allows us to integrate by parts to get \r{7} in this case and choose $N\sim \lambda/(Ce)$ to make the r.h.s.\ of \r{7} exponentially small with $\lambda$. 
The phase function $\Phi$ has an analytic extension for $\zeta$ in  some complex neighborhood of $\zeta_0=e^{n-1}$. By Lemma~\ref{lemma_phase}, $\phi_\zeta(z,\zeta)=\phi_\zeta(w,\zeta)$ for such $\zeta$ and  $z$, $w$ close to $0$ implies $z=w$. Therefore, the critical point $\zeta_c=\eta$ of $\Phi$ w.r.t.\ $\zeta$ is real only when $z=w$ and at that point, $\Im\Phi_{\zeta\zeta}>0$; and it is unique and complex otherwise, still satisfying that inequality by a perturbation argument, when $0<\delta\ll\eps$. Then we get \r{8} with $\psi$ defined as in \r{psi}.  Then we conclude as in \cite{FSU, SU-AJM} that $(0,\xi_0)\not\in\WFA(f)$. 

This arguments so far work if $f$ is a continuous function, for example. If $f$ is a distribution, as stated, we need to take a smooth cutoff $\chi_\delta$ and consider the $z$-integrals above in distribution sense. 
\end{proof}

\begin{proof}[Proof of Theorem~\ref{thm_L}] 
Assume that the statement of the theorem is not true.  
Let $\sigma_0\in [0,1)$ be the infimum  of all $\sigma$ for which $F^{-1}(\sigma)\cap \supp f\not=\emptyset$. 
Then $f=0$ in the ``exterior'' $F^{-1}(0,\sigma_0)$ of $S_0= F^{-1}(\sigma_0)$, and  $\supp f$ has a common $z_0$ point with $S_0$. The latter follows from a compactness argument. In particular, $\sigma_0>0$. 
Since $S_0$ is timelike by assumption, there is a lightlike geodesic $\gamma_0$ through $z_0$ tangent to $S_0$ which does not hit $S_0$ again by the strict convexity assumption. Near $z_0$, the geodesic $\gamma_0$ lies in the exterior $F^{-1}[0,\sigma_0]$ by the local part of the strict convexity assumption; and this is also true globally by the global part of that assumption. 
By Theorem~\ref{thm_LC}, $f=0$ near every common point of $S_0$ and $\supp f$, which is a contradiction. 
\end{proof}


\end{document}